\documentclass[12pt,twoside]{amsart}
\usepackage{amssymb}
\usepackage{amscd}

\usepackage{hyperref}

\title[Smooth rational surfaces]
{Smooth rational surfaces violating Kawamata--Viehweg vanishing} 
\author{Paolo Cascini and Hiromu Tanaka} 

\address{Department of Mathematics, Imperial College, London, 180 Queen's Gate, 
London SW7 2AZ, UK} 
\address{National Center for Theoretical Sciences
No. 1 Sec. 4 Roosevelt Rd.,
National Taiwan University
Taipei, 106, Taiwan}
\email{p.cascini@imperial.ac.uk}
\address{Department of Mathematics, Imperial College, London, 180 Queen's Gate, 
London SW7 2AZ, UK} 
\email{h.tanaka@imperial.ac.uk}
\thanks{Both of the authors were funded by EPSRC}

\newcommand{\Spec}[0]{{\operatorname{Spec}}}

\newcommand{\Ex}[0]{{\operatorname{Ex}}}
\newtheorem{thm}{Theorem}[section]
\newtheorem{lem}[thm]{Lemma}
\newtheorem{cor}[thm]{Corollary}
\newtheorem{prop}[thm]{Proposition}

\theoremstyle{definition}

\newtheorem{rem}[thm]{Remark}

\newtheorem{nota}[thm]{Notation}         
\newtheorem{step}{Step}

\makeatletter
  
  \@addtoreset{equation}{section}
  \makeatother

\newcommand{\MO}{\mathcal{O}}

\newcommand{\Q}{\mathbb{Q}}
\newcommand{\Z}{\mathbb{Z}}
\newcommand{\F}{\mathbb{F}}

\begin{document}

\maketitle

\begin{abstract}
We show that over any algebraically closed field of positive characteristic, 
there exists a smooth rational surface which violates Kawamata--Viehweg vanishing. 
\end{abstract}

\tableofcontents

\setcounter{section}{0}

\section{Introduction}

It is a well known fact that Kodaira vanishing fails in positive characteristic \cite{Ray78}. Nevertheless, 
it has often been  believed that a stronger version, namely Kawamata--Viehweg vanishing,
holds over a smooth rational surface (e.g. see \cite{ter98,xie10}). 

In this note, we show that this is indeed not true: 
%

\begin{thm}[Theorem~\ref{t-cex-kvv}]\label{intro-kvv}
Let $k$ be a field of positive characteristic. 
Then there exist a smooth projective rational surface $X$ over $k$, a Cartier divisor $D$, 
and a $\Q$-divisor $\Delta\ge 0$ such that
\begin{enumerate}
\item $(X, \Delta)$ is klt, 
\item $D-(K_X+\Delta)$ is nef and big, and  
\item $H^1(X, \MO_X(D)) \neq 0$. 
\end{enumerate}
\end{thm}

To prove Theorem~\ref{intro-kvv}, 
we  use some surfaces constructed by Langer \cite{Lan}. 
If $k=\F_p$, then $X$ can be obtained 
by taking the blowup of $\mathbb P^2_{\F_p}$ along all the $\F_p$-rational points. 
Since the proper transforms $L'_1, \dots, L'_{p^2+p+1}$ of the $\F_p$-lines $L_1, \dots, L_{p^2+p+1}$ 
are pairwise disjoint, 
we can contract all  these curves and obtain a birational morphism $g\colon X \to Y$ onto a klt surface $Y$ such that  $\rho(Y)=1$ (cf. Lemma~\ref{l-Y-sing}). 
Note that  $-K_Y$ is ample if and only if $p=2$ (cf. Lemma~\ref{l-Y-sing}). 
Further, we show:

\begin{itemize}
\item For any $p>0$, $Y$ is obtained as a purely inseparable cover of $\mathbb P^2$
 (cf. Theorem~\ref{t-pi}). 
If $p=2$, then the morphism $Y \to \mathbb P^2$ is induced by the anti-canonical linear system $|-K_Y|$ 
(cf. Remark~\ref{r-anti-map}).  
\item If $p=2$, then the Kleimann--Mori cone $NE(X)$ is generated by exactly 14 curves 
(cf. Theorem~\ref{t-KM-cone}). 
\item If $p=2$, then $X$ is isomorphic to a surface constructed by Keel--M\textsuperscript{c}Kernan 
(cf. Proposition~\ref{p-km-surfaces}). 
\end{itemize}

\medskip

\textbf{Related results:}
After Raynaud constructed  the first counter-example 
to  Kodaira vanishing  in positive characteristic \cite{Ray78}, 
several other people studied this problem 
(e.g. see \cite{DI87},  \cite{Eke88}, \cite{She91}, \cite[Section~II. 6]{Kol96}, \cite{Muk13}, \cite{DF15} ). 
In particular,  Fano varieties are known to violate  Kawamata--Viehweg vanishing. 
As far as the authors know, the examples constructed by Lauritzen--Rao \cite{LR97} (of dimension at least $6$) are the only one 
over an algebraically closed field. 
If we admit imperfect fields, 
then Schr\"{o}er and Maddock constructed  log del Pezzo surfaces with $H^1(X, \MO_X) \neq 0$ 
\cite{Sch07,Mad16}. 
In \cite{CTW}, the authors and Witaszek show that Kawamata--Vieweg vanishing holds  
for klt del Pezzo surfaces in large characteristic. 
On the other hand, if $p=2$, then the surface mentioned above is a smooth weak del Pezzo surface (cf. Lemma~\ref{l-Y-sing}), 
hence our result cannot be extended to characteristic two (see also Proposition~\ref{p-A}).

\medskip

\textbf{Acknowledgement:} 
We would like to thank I. Cheltsov, K. Fujita, A. Langer and J. Witaszek for many useful discussions and comments. 
We would also like to thank the referee for reading our manuscript carefully and for suggesting several improvements.

\section{Preliminaries}
 
\subsection{Notation}

We say that $X$ is a {\em variety} over a field $k$, 
if $X$ is an integral  scheme which is separated and of finite type over $k$. 
A {\em curve} (resp. {\em surface}) 
is a variety of dimension one (resp. two). 
We say that two schemes $X$ and $Y$ over a field $k$ are $k$-{\em isomorphic} 
if there exists an isomorphism $\theta\colon X \to Y$ of schemes 
such that both $\theta$ and $\theta^{-1}$ commutes with the structure morphisms: 
$X \to \Spec\,k$ and $Y \to \Spec\,k$. 
Given a proper morphism $f\colon X\to Y$ between normal varieties, we say that  two 
$\mathbb Q$-Cartier $\mathbb Q$-divisors 
$D_1,D_2$ on $X$ are {\em numerically equivalent over} $Y$, denoted $D_1\equiv_f D_2$, if their difference is numerically trivial on any fibre of $f$.  

We refer to \cite[Section 2.3]{KM98} or \cite[Definition~2.8]{Kol13} for the classical definitions of singularities 
(e.g. {\em klt}) appearing in the minimal model programme. 
Note  that we always assume that 
for any klt 
pair $(X, \Delta)$, 
the $\Q$-divisor $\Delta$ is effective. 

\subsection{Construction by Langer}\label{s-construct}

We now recall the construction of a rational surface due to Langer \cite{Lan} (see also \cite[Exercise III.10.7]{Har77}). 
A similar method was used to construct  also some  K3 surfaces and  Calabi--Yau threefolds (cf. \cite{Hir99}, \cite{DK03}). 
 
\begin{nota}\label{n-q}
\begin{enumerate}
\item 
Let $q:=p^e$, where $p$ is a prime number and $e$ is a positive integer. 
Let $P_1^{(0)}, \dots, P_{q^2+q+1}^{(0)}$ be the $\F_q$-rational points on $\mathbb P^2_{\F_q}$, 
and let $L_1^{(0)}, \dots, L_{q^2+q+1}^{(0)}$ be the $\F_q$-lines on $\mathbb P^2_{\F_q}$, 
i.e. the lines which are defined over $\F_q$. 
Let 
$$f^{(0)}\colon X^{(0)} \to \mathbb P^2_{\F_q}$$
be the blowup along all the $\F_q$-points $P_1^{(0)}, \dots, P_{q^2+q+1}^{(0)}$. 
For any $i=1, \dots, q^2+q+1$, let $E^{(0)}_i$ be the $f^{(0)}$-exceptional prime divisor 
lying over $P^{(0)}_i$, hence $E^{(0)}_i \simeq \mathbb P^1_{\F_q}$. 
The proper transforms $L'^{(0)}_1, \dots, L'^{(0)}_{q^2+q+1}$ of the $\F_q$-lines are disjoint each other 
and satisfy $(L'^{(0)}_i)^2=-q$ for any $i=1, \dots, q^2+q+1$. 
Let 
$$g^{(0)}\colon X^{(0)} \to Y^{(0)}$$
be the birational morphism contracting all of the curves 
$$L'^{(0)}_1, \dots, L'^{(0)}_{q^2+q+1}.$$ 
We define $(E_i^Y)^{(0)}:=g^{(0)}_*E_i^{(0)}$. 
\item 
Let $k$ be a field containing $\F_q$ and let  
$$f\colon X \to \mathbb P^2_k, \quad g\colon X \to Y$$
be the base changes of $f^{(0)}$ and $g^{(0)}$ induced by $(-) \times_{\F_q} k$. 
We denote by $P_i$, $L_i$, $E_i$, $L'_i$ and $E^Y_i$ 
 the inverse images of $P^{(0)}_i$, $L^{(0)}_i$, $E^{(0)}_i$, $L'^{(0)}_i$ and $(E^Y_i)^{(0)}$ , respectively. 
We fix an arbitrary line $H \in |\MO_{\mathbb P^2}(1)|$ defined over $k$. 
By abuse of notation, 
each $P_i$ (resp. $L_j$) is also called an $\F_q$-point (resp. an $\F_q$-line), 
although these depend on the choice of the homogeneous coordinates. 
\end{enumerate}
\end{nota}

\begin{nota}\label{n-2}
We use the  same notation as in Notation~\ref{n-q} but we assume that $q=2$, i.e. $p=2$ and $e=1$. 
\end{nota}

\begin{rem}\label{r-configuration}
The configuration of the $\F_q$-points and the $\F_q$-lines on $\mathbb P^2_{\F_q}$ satisfy 
the following properties:
\begin{itemize}
\item For any $\F_q$-line $L$ on $\mathbb P^2_{\F_q}$, 
the number of the $\F_q$-points contained in $L$ is equal to $q+1$. 
\item For any $\F_q$-point $P$ on $\mathbb P^2_{\F_q}$, 
the number of the $\F_q$-lines passing through $P$ is equal to $q+1$. 
\end{itemize}
If $q=2$, then the picture of the configuration is classically known as Fano plane 
(e.g. see \cite[Subsection~3.1.1]{Pol98}). 
\end{rem}

\subsection{Basic properties}

We now summarise some basic properties of the surfaces $X$ and $Y$ constructed in Notation~\ref{n-q}. 

\begin{lem}\label{l-Y-sing}
We use Notation~\ref{n-q}. 
The following  hold: 
\begin{enumerate}
\item $\rho(Y)=1$. 
\item $Y$ is klt. 
\item $Y$ has at most canonical singularities if and only if $q=2$. 
\item If $q>2$, then $K_Y$ is ample. 
\item If $q=2$, then $-K_Y$ is ample. 
\item If $q=2$, then $-K_X$ is nef and big. 
\end{enumerate}
\end{lem}

\begin{proof} (1) follows immediately by the construction. 

Further, we have 
$$g^*K_Y=K_X+\left (1-\frac 2 q\right )  \sum_{i=1}^{q^2+q+1} L'_i.$$
Thus, (2) and (3) hold. 

We now show (4) and (5). 
Since $K_X=f^*K_{\mathbb P^2}+\sum_i E_i\sim -3f^*H+\sum_i E_i$ and 
$$(q^2+q+1)f^*H \sim f^*(\sum_{i=1}^{q^2+q+1}L_i)=\sum_{i=1}^{q^2+q+1}L'_i+(q+1)\sum_{i=1}^{q^2+q+1}E_i,$$
we have that 
\begin{eqnarray*}
(q^2+q+1)K_X&\sim& -3(q^2+q+1)f^*H+(q^2+q+1)\sum_{i=1}^{q^2+q+1} E_i\\
&\sim& -3\sum_{i=1}^{q^2+q+1}L'_i+(q^2-2q-2)\sum_{i=1}^{q^2+q+1} E_i.
\end{eqnarray*}
Taking the push-forward $g_*$, we get 
$$(q^2+q+1)K_Y \sim (q^2-2q-2)\sum_{i=1}^{q^2+q+1} E^Y_i.$$
Therefore, if $q=2$ (resp. $q >2$), then $-K_Y$ (resp. $K_Y$) is ample. 
Thus, (4) and (5) hold. 
 (6)  follows directly from (3) and (5). 
\end{proof}

\begin{lem}\label{l-lines-cover}
We use Notation~\ref{n-q}. We assume that $k=\F_q$. 
For any $\F_q$-point $P_i \in \mathbb P_{\F_q}^2(\F_q)$, 
let $L_{j_1}, \dots, L_{j_{q+1}}$ be the $\F_q$-lines passing through $P_i$. 
Then $\mathbb P_{\F_q}^2(\F_q)=L_{j_1}(\F_q) \cup \dots \cup L_{j_{q+1}}(\F_q)$. 
\end{lem}

\begin{proof}
Since we have $L_{j_{\alpha}} \cap L_{j_{\beta}}=P_i$ for any $1 \leq \alpha < \beta \leq q+1$, 
the claim follows by counting the number of $\F_q$-rational points (cf. Remark~\ref{r-configuration}): 
$$\#(L_{j_1}\cup \dots \cup L_{j_{q+1}})(\F_q)=q (q+1)+1=q^2+q+1=\mathbb P^2_{\F_q}(\F_q).$$ 
\end{proof}


\section{Counter-examples to  Kawamata--Viehweg vanishing}\label{s-kvv}


In this Section, we construct some counter-examples to  Kawamata--Viehweg vanishing  on a family of  
 smooth rational surfaces:

\begin{thm}\label{t-cex-kvv}
We use Notation~\ref{n-q}. 
We consider the following $\mathbb Q$-divisors on $X$:  
\begin{itemize}
\item $\Delta:=\frac{q}{q+1}\sum_{i=1}^{q^2+q+1}L'_i$, and 
\item $B:=(q^2+1)f^*H-q\sum_{i=1}^{q^2+q+1} E_i.$
\end{itemize}
Then the following  hold:
\begin{enumerate}
\item $(X, \Delta)$ is klt. 
\item $B-\Delta$ is nef and big. 
\item $h^1(X, \MO_X(K_X+B)) \geq \frac{1}{2}(q^2-q)$. 
\end{enumerate}
In particular, Kawamata-Viehweg vanishing fails on $X$. 
\end{thm}

\begin{proof}
Since $L'_1, \dots, L'_{q^2+q+1}$ are pairwise disjoint, (1) follows immediately. 
We now show (2). 
We have:
$$(q^2+q+1)f^*H \sim f^*(\sum_{i=1}^{q^2+q+1} L_i)
=\sum_{i=1}^{q^2+q+1} L'_i+(q+1)\sum_{i=1}^{q^2+q+1} E_i.$$
It follows that 
$$B=(q^2+1)f^*H-q\sum_{i=1}^{q^2+q+1} E_i \sim_{\Q} \frac{1}{q+1}f^*H+\frac{q}{q+1}\sum_{i=1}^{q^2+q+1}L'_i.$$
Thus, (2) holds. 

We now show (3). 
By  Riemann--Roch, it follows that  
$$\chi(X, \MO_X(K_X+B))=1+\frac{1}{2}(B^2+B \cdot K_X).$$
Since 
$$B^2=((q^2+1)f^*H-q\sum_{i=1}^{q^2+q+1} E_i)^2=(q^2+1)^2-q^2(q^2+q+1)=-q^3+q^2+1$$
and 
$$B \cdot K_X=((q^2+1)f^*H-q\sum_{i=1}^{q^2+q+1} E_i) \cdot (-3f^*H+\sum_{i=1}^{q^2+q+1} E_i)$$
$$=-3(q^2+1)+q(q^2+q+1)=q^3-2q^2+q-3,$$
we have that 
$$\chi(X, K_X+B)=1+\frac{1}{2}((-q^3+q^2+1)+(q^3-2q^2+q-3))=\frac{1}{2}(-q^2+q).$$
Thus, (3) holds. 
\end{proof}

\begin{rem} We do not know whether 
there exist a klt del Pezzo surface $X$ and a nef and big Cartier divisor $A$ on $X$ such that $$H^1(X, \MO_X(A)) \neq 0.$$ 
\end{rem}

As an application, we now show that the pair $(X, \sum E_i+\sum L'_j)$ is not liftable to $W_2(k)$. Note that, a similar result was proven in \cite[Proposition 8.4]{Lan}.

\begin{cor}
We use Notation~\ref{n-q}. Assume that $k$ is perfect. 
If $p \geq 3$, then 
$$(X, \sum_{i=1}^{q^2+q+1} E_i+\sum_{j=1}^{q^2+q+1} L'_j)$$ 
is not liftable to $W_2(k)$. 
\end{cor}

\begin{proof}
We use the same notation as in Theorem~\ref{t-cex-kvv}. 
As in the proof of Theorem~\ref{t-cex-kvv}, 
it follows that $B-\Delta-\sum \epsilon_iE_i$ is ample for some $\epsilon_i>0$. 
Thus, Theorem~\ref{t-cex-kvv} and \cite[Corollary~3.8]{Har98} imply the claim. 
\end{proof}

\section{Purely inseparable morphisms to $\mathbb P^2$}\label{s-pi}

The main purpose of this Section is to show that the  surface $Y$, as in Notation~\ref{n-q},
can be obtained as a purely inseparable cover of $\mathbb P^2$ (cf. Theorem~\ref{t-pi}). 
Moreover if $q=2$, then the morphism $Y \to \mathbb P^2$ is induced 
by the anti-canonical linear system (cf. Remark~\ref{r-anti-map}). 

We also show that the complete linear system $|M|$, appearing in Theorem~\ref{t-pi}, 
does not have any smooth element (cf. Proposition~\ref{p-bertini}), even though it is base point free and big. We were not able to find a similar example in the literature 
 (cf. \cite[Theorem II.8.18 and Corollary III.10.9]{Har77}).

\begin{thm}\label{t-pi}
We use Notation~\ref{n-q}. 
Let  
$$M:=(q+1)f^*H-\sum_{i=1}^{q^2+q+1} E_i.$$
Then the following  hold:
\begin{enumerate}
\item $|M|$ is base point free. 
\item $M \cdot L'_j=0$ for any $j=1, \dots, q^2+q+1$. 
\item $M^2=q$. 
\item 
Given the natural injective $k$-linear map  
$$\iota\colon H^0(X, \MO_X(M)) \hookrightarrow H^0(\mathbb P^2_k, \MO_{\mathbb P^2_k}(q+1)),$$ 
the following  holds: 
$$\iota(H^0(X, \MO_X(M)))=k\cdot (x^qy-xy^q)+k\cdot (y^qz-yz^q)+k\cdot (z^qx-zx^q).$$
\item There exists a Cartier divisor $M_Y$ on $Y$ such that $M=g^*M_Y$. 
\item The morphism induced by the complete linear system $|M_Y|$ 
$$\varphi:=\Phi_{|M_Y|}\colon Y \to \mathbb P_k^2$$
is a finite universal homeomorphism of degree $q$. 
\end{enumerate}
\end{thm}

\begin{proof}
We may assume that $k=\F_q$. 

We first show (1). 
Given a $\F_q$-point $P_i$ on $\mathbb P^2_{\F_q}$, we denote by  $L_{j_1}, \dots, L_{j_{q+1}}$  
the $\F_q$-lines passing through $P_i$. Then  Lemma~\ref{l-lines-cover} implies that
$$M=(q+1)f^*H-\sum_{r=1}^{q^2+q+1} E_r \sim \sum_{\alpha=1}^{q+1}f^*L_{j_{\alpha}}-\sum_{r=1}^{q^2+q+1} E_r=qE_i+\sum_{\alpha=1}^{q+1}L'_{j_{\alpha}}.$$
Thus, $|M|$ is base point free by symmetry and (1) holds. 

(2) and (3) are simple calculations,
and (4) follows from \cite{tal61a,tal61b} (see also  \cite[Proposition 2.1]{HK13}
\footnote{Note that we cite the arXiv version, as the published version omits the proof of \cite[Proposition 2.1]{HK13}.}).  
Further, $g\colon X \to Y$ is the Stein factorisation of $\psi:=\Phi_{|M|}:X \to \mathbb P^2_k$. 
Thus, (5) holds. 

We now show (6). 
Since $M=g^*M_Y$, (1) implies  that 
$|M_Y|$ is base point free and (5) implies that $h^0(Y, \MO_Y(M_Y))=3$.
Since $M_Y$ is ample, it follows that $\varphi$ is a finite surjective morphism. 
By (3), the degree of $\varphi$ is equal to $q$.

It is enough to show that $\varphi$ is a purely inseparable morphism. 
To this end, we may assume that $k=\overline{\F}_q$. 
By (4), we have that 
$$\psi \circ f^{-1}\colon \mathbb P_k^2 \dashrightarrow \mathbb P^2_k, 
\quad [x:y:z] \mapsto [x^qy-xy^q:y^qz-yz^q:z^qx-zx^q].$$
Generically, 
the rational map $\psi \circ f^{-1}$ can be written by 
$$\Psi\colon \mathbb A_k^2 \setminus \bigcup_{i=1}^{q+1}\widetilde L_i \to \mathbb A^2_k, 
\quad (u, v) \mapsto \left(\frac{v^q-v}{u^qv-uv^q}, \frac{u-u^q}{u^qv-uv^q}\right),$$
where $\widetilde L_1, \dots, \widetilde L_{q+1}$ are the affine lines passing through the origin 
with  coefficients in $\F_q$, and in particular $\bigcup_{i=1}^{q+1}\widetilde L_i=\{u^qv-uv^q=0\}$. 
Fix a general closed point $(\alpha, \beta) \in \mathbb A^2_k$.
It is enough to show that its fibre $\Psi^{-1}((\alpha, \beta))$ consists of one point. 
Let $(u, v) \in \mathbb A_k^2 \setminus \bigcup_{i=1}^{q+1}\tilde L_i$  be 
such that $\Psi(u, v)=(\alpha, \beta)$. 
Since $(\alpha, \beta)$ is chosen to be general, we can assume that the denominators of 
the fractions appearing in the following calculation are always nonzero.
We have that 
$$\alpha(u^qv-uv^q)=v^q-v, \quad \beta(u^qv-uv^q)=u-u^q,$$
which implies 
\begin{equation}\label{e-pi}
\alpha(u^q-uv^{q-1})=v^{q-1}-1,
\end{equation}
and 
\begin{equation}\label{e-pi2}
\beta(u^{q-1}v-v^q)=1-u^{q-1}.
\end{equation}
By (\ref{e-pi}), we have that 
\begin{equation}\label{e-pi3}
v^{q-1}=\frac{\alpha u^q+1}{\alpha u+1}.
\end{equation}
Substituting (\ref{e-pi3}) to (\ref{e-pi2}), we get 
\begin{equation}\label{e-pi4}
v=\frac{1}{\beta}\cdot\frac{1-u^{q-1}}{u^{q-1}-v^{q-1}}=\frac{1}{\beta}\cdot\frac{1-u^{q-1}}{u^{q-1}-\frac{\alpha u^q+1}{\alpha u+1}}=-\frac{\alpha u+1}{\beta}.
\end{equation}
Substituting (\ref{e-pi4}) to (\ref{e-pi3}), it follows that 
$$
\alpha u^q+1=(\alpha u+1)v^{q-1}=(\alpha u+1)\left(-\frac{\alpha u+1}{\beta}\right)^{q-1}
=\frac{(-1)^{q-1}(\alpha^q u^q +1)}{\beta^{q-1}},
$$
which implies that 
$$
u^q=\frac{-\beta^{q-1}+(-1)^{q-1}}{\alpha\beta^{q-1}-(-1)^{q-1}\alpha^q}.
$$
This implies that $u$ is uniquely determined by $(\alpha, \beta)$, 
and so is $v$ by (\ref{e-pi4}). 
Thus, (6) holds. 
\end{proof}

\begin{rem}\label{r-anti-map}
Using the  same notation as in Theorem~\ref{t-pi}, 
if $q=2$, then $M=-K_X$ and $M_Y=-K_Y$. 
This can be considered as an analogue 
of the fact that 
a smooth del Pezzo surface $S$ with $K_S^2=2$ is a double cover of $\mathbb P^2$ 
which is induced by the anti-canonical system $|-K_X|$. Indeed, 
both  $X$ and $S$ are obtained by taking blowups along seven points. 
\end{rem}

\begin{prop}\label{p-bertini}
We use Notation~\ref{n-q}. 
Let  
$$M:=(q+1)f^*H-\sum_{i=1}^{q^2+q+1} E_i.$$
Then the following  hold:
\begin{enumerate}
\item If $k=\F_q$, then for any element $D \in |M|$, 
there exists a unique $\F_q$-point $P_i$ on $\mathbb P^2_{\F_q}$ 
such that 
$$D=qE_i+\sum_{\alpha=1}^{q+1}L'_{j_{\alpha}}$$
where $L_{j_1}, \dots, L_{j_{q+1}}$ are the $\F_q$-lines passing through $P_i$. 
\item If $k$ is an algebraically closed field, then a general member of $|M|$ is integral. 
\item Any element of $|M|$ is not smooth. 
\end{enumerate}
\end{prop}

\begin{proof}
Note that for each $\F_q$-point  $P_i$ on $\mathbb P^2_{\F_q}$, 
the divisor $D=qE_i+\sum_{\alpha=1}^{q+1}L'_{j_{\alpha}}$, as in (1), 
is an element of $|M|$. Thus, there are $q^2+q+1$ of such divisors. 
On the other hand, (4) of Theorem~\ref{t-pi} implies
$$\#|M|=\frac {q^3-1}{q-1}.$$
Thus, (1) holds (see also \cite[Proposition~2.3]{HK13}).

We now show (2) and (3). 
To this end, we may assume that $k$ is algebraically closed. 
We set $M_Y:=g_*M$. 
By (1), there exists an irreducible divisor in $|M_Y|$. 
Thus, any  general element of $|M_Y|$ is irreducible. 

Since, by Theorem~\ref{t-pi},  $|M_Y|$ is base point free, if $D\in |M|$ is a general element, then $D$ is irreducible. 
By Theorem~\ref{t-pi}, we may write 
$$f_*D=\{\gamma(x^qy-xy^q)+\alpha(y^qz-yz^q)+\beta(z^qx-zx^q)=0\}$$
for some $(\alpha, \beta, \gamma) \in k^3 \setminus \{(0, 0, 0)\}$. 
By the Jacobian criterion for smoothness, 
it follows that $[\alpha^{1/q}:\beta^{1/q}:\gamma^{1/q}]$ is a unique singular point of $f_*D$. 
Since $f_*D$ is smooth outside $[\alpha^{1/q}:\beta^{1/q}:\gamma^{1/q}]$, 
we see that $f_*D$ is reduced. 
Since $\alpha, \beta, \gamma$ are chosen to be general, 
it follows that $[\alpha^{1/q}:\beta^{1/q}:\gamma^{1/q}]$ is not an $\F_q$-point. 
Thus, $D$ is the proper transform of $f_*D$, hence $D$ is integral. 
Thus, (2) holds. 
Since $f_*D$ has a singular point outside $f(\Ex(f))$, it follows that $D$ is not smooth. 
Thus, (3) holds. 
\end{proof}


\section{The Kleimann--Mori cone}\label{s-km-cone}

The main result of this Section is Theorem~\ref{t-KM-cone} 
which determines the generators of the Kleimann--Mori cone of $X$ as in Notation~\ref{n-2}. 
To this end, we classify the curves whose self-intersection numbers are negative 
(cf. Proposition~\ref{p-negative-curves}).

\begin{lem}\label{l-negative}
We use Notation~\ref{n-2}. 
The following  hold:
\begin{enumerate}
\item If $C$ is a curve on $X$ which satisfies $C^2=-1$ and differs from any of $E_1, \dots, E_7$, 
then 
$$\deg(f_*(C)) \leq 3.$$ 
\item  If $C$ is a curve on $X$ with $C^2=-2$, 
then $$\deg(f_*(C)) \leq 2.$$
\end{enumerate}
\end{lem}

\begin{proof}
We show (1). 
We have 
$$C \sim af^*\MO_{\mathbb P^2}(1)+ \sum_{i=1}^7 b_iE_i$$
where  $a=\deg(f_*(C))> 0$ and  $b_1,\dots,b_7 \in \Z$. 
Since $q=2$, Lemma~\ref{l-Y-sing} implies that 
$C$ is a $(-1)$-curve. Thus,  we have 
\begin{eqnarray*}
-1=C^2&=&a^2- \sum_{i=1}^7 b_i^2\\
-1=K_X \cdot C&=&(-3f^*H+\sum_{i=1}^7 E_i)\cdot (af^*H+ \sum_{i=1}^7 b_iE_i)
=-3a-\sum_{i=1}^7 b_i.\\
\end{eqnarray*}
By Schwarz's inequality, we obtain 
$$(3a-1)^2=\left(\sum_{i=1}^7b_i\right)^2 \leq 7\sum_{i=1}^7 b_i^2=7(a^2+1),$$
which implies $a^2-3a-3 \leq 0.$ 
Thus, (1) holds. 
The proof of (2) is   similar. 
\end{proof}

\begin{lem}\label{l-conic-cubic}
We use Notation~\ref{n-2}. 
Let $C$ be a curve on $X$ such that $C_0:=f(C)$ is a conic or a cubic. 
Then $C^2 \geq 0$. 
\end{lem}

\begin{proof}
First, we assume that $C_0$ is conic. 
Suppose that $C_0$ passes through 5 of the $\F_2$ -points, say $P_1, \dots, P_5$. 
Let us derive a contradiction. 
Let $P_6$ and $P_7$ be the remaining two $\F_2$-points. 
Since there are exactly three $\F_2$-lines passing through $P_6$ (resp. $P_7$), 
we can find an $\F_2$-line $L_i$ such that $P_6 \not\in L_i$ and $P_7 \not\in L_i$. 
In particular, $C_0 \cap L_i$ contains at least three points, within $P_1, \dots, P_5$. 
This contradicts the fact that $C_0 \cdot L_i=2$. 

Now, we assume that $C_0$ is cubic. 
If $C_0$ is smooth, then  $C^2 \geq C^2_0-7=2$. 
Thus, we may assume that $C_0$ is singular and $C^2<0$. 
It follows that $C_0$ must pass through all the $\F_2$-points $P_1, \dots, P_7$
and the unique singular point of $C_0$ is an $\F_2$-point, say $P_1$. 
Let $L_j$ be an $\F_2$-line passing through $P_1$. 
Since $C_0 \cap L_j$ contains at least three $\F_2$-rational points $P_1, P_{i}, P_{i'}$, 
we have that $C_0 \cdot L_j \geq 4$. 
This contradicts the fact that $C_0 \cdot L_j=3$. 
Thus, the claim follows. 
\end{proof}


\begin{prop}\label{p-negative-curves}
We use Notation~\ref{n-2}. 
Let $C$ be a curve on $X$ with $C^2<0$. 
Then $C$ is equal to one of the curves $E_1, \dots, E_7, L'_1, \dots, L'_7$. 
\end{prop}

\begin{proof}
Assume that $C\notin\{E_1, \dots, E_7\}$. 
Let $C_0:=f_*C$. 
Since $-K_X$ is nef and big, we have that $C^2\geq -2$.  
Lemma~\ref{l-negative} implies that $\deg C_0 \leq 3$. 
By Lemma~\ref{l-conic-cubic}, we have that $\deg C_0=1$, hence $C_0$ is a line. 
Then $C_0$ passes through at least two of the $\F_2$-points. 
It follows that $C_0$ is equal to some $L_i$, hence $C=L'_i$, as desired. 
\end{proof}

\begin{thm}\label{t-KM-cone}
We use Notation~\ref{n-2}. 
Then 
$$\overline{NE}(X)=NE(X)=\sum_{i=1}^7\mathbb R_{\geq 0}[E_i]+\sum_{j=1}^7\mathbb R_{\geq 0}[L'_j].$$
\end{thm}

\begin{proof}
Since there exists an effective $\Q$-divisor $\Delta$ 
such that $(X, \Delta)$ is klt and $-(K_X+\Delta)$ is ample, the cone theorem \cite[Theorem~1.7]{Tana} implies that  
$NE(X)$ is closed and  generated by the extremal rays spanned by curves. 
By \cite[Theorem~4.3]{Tanb}, any extremal ray of $NE(X)$ is generated by a curve $C$ 
whose self-intersection number is negative. 
Thus, the claim follows from Proposition~\ref{p-negative-curves}. 
\end{proof}

\section{Relation to Keel--M\textsuperscript{c}Kernan surfaces}\label{s-km-surfaces}

The goal of this Section is to prove Proposition~\ref{p-km-surfaces}  which shows that the surface $X$, constructed in Notation~\ref{n-2}, is isomorphic 
to some surface obtained by Keel--M\textsuperscript{c}Kernan \cite[end of \S 9]{KM99}. 

We first recall their construction. 
Let $k$ be a field of characteristic two. 
We fix a $k$-rational point in $\mathbb P^2_k$ and a conic over $k$ as follows: 
$$Q:=[0:0:1] \in \mathbb P^2_k, \quad C:=\{xy+z^2=0\} \subset \mathbb P^2_k.$$
Note that any line through $Q$ is tangent to $C$. 
Let $\varphi_0\colon S_0 \to \mathbb P^2_k$ be the blowup at $Q$. 
We choose $k$-rational points $P_1, \dots, P_d$ at $\varphi_0^{-1}(C)$. 
We first consider the blowup along these points $\psi\colon S'_0 \to S_0$ and 
then we  take the blowup $S \to S'_0$ along the intersection $\Ex(\psi) \cap \psi_*^{-1}(\varphi^{-1}(C))$, 
where $\psi_*^{-1}(\varphi_0^{-1}(C))$ is the proper transform of $\varphi^{-1}(C)$. 
Note that the intersection $\Ex(\psi) \cap \psi_*^{-1}(\varphi^{-1}(C))$ is a collection of $k$-rational points. 
We call $S$ a {\em Keel--M\textsuperscript{c}Kernan surface} of degree $d$ over $k$. 

\medskip

Let us recall a well-known result on the theory of Severi--Brauer varieties. 

\begin{lem}\label{l-SB}
Let $X$ be a projective scheme over $\F_q$. 
Let $\overline{\F}_q$ be the algebraic closure of $\F_q$. 
If the base change $X \times_{\F_q} \overline{\F}_q$ is $\overline{\F}_q$-isomorphic to $\mathbb P^n_{\overline{\F}_q}$, 
then $X$ is $\F_q$-isomorphic to $\mathbb P^n_{\F_q}$. 
\end{lem}

\begin{proof}
See, for example, \cite[Chap. X, \S 5, \S 6, \S 7]{Ser79}. As an alternative proof, one can conclude the claim 
from \cite[Corollary~1.2]{Esn03} and Ch\^atelet's theorem \cite[Theorem~5.1.3]{GS06}. 
\end{proof}

The following two lemmas may be well-known, 
however we include proofs for the sake of completeness. 

\begin{lem}\label{l-P2-auto}
Let $k$ be a field. 
Take $k$-rational points 
$$P_1, \cdots, P_4, Q_1, \cdots, Q_4 \in \mathbb P^2_k.$$ 
Assume that no three of $P_1, \cdots, P_4$ (resp. $Q_1, \cdots, Q_4$) 
lie on a single line of $\mathbb P^2_k$. 
Then there exists a $k$-automorphism $\sigma\colon \mathbb P^2_k \to \mathbb P^2_k$ such that $\sigma(P_i)=Q_i$ 
for any $i \in \{1, 2, 3, 4\}$. 
\end{lem}

\begin{proof}
We may assume that 
$$P_1=[1:0:0], \quad P_2=[0:1:0], \quad P_3=[0:0:1], \quad P_4=[1:1:1].$$ 
For each $i \in \{1, 2, 3, 4\}$, we  write $Q_i=[a_i:b_i:c_i]$ for some $a_i, b_i, c_i \in k$. 
Consider the matrix 
$$M:= \left(\begin{array}{ccc} a_1 & a_2 & a_3 \\ b_1 & b_2 & b_3\\ c_1 & c_2 & c_3 \\\end{array} \right).$$
Since  $Q_1, Q_2, Q_3$ do not lie on a  line, it follows that $\det M \neq 0$. 
Let $\tau\colon \mathbb P^2_k \to \mathbb P^2_k$ be the $k$-automorphism induced by $M$. 
In particular,
$$\tau([1:0:0])=Q_1, \quad \tau([0:1:0])=Q_2, \quad \tau([0:0:1])=Q_3.$$
We may write $\tau^{-1}(Q_4)=[d:e:f]$ for some $d, e, f \in k$. 
Again by the assumption, we have that $d,e,f\neq 0$. 
Then the $k$-automorphism 
$$\rho\colon \mathbb P^2_k \to \mathbb P^2_k, \quad [x:y:z]\mapsto [dx:ey:fz]$$ 
satisfies 
$$\rho([1:0:0])=[1:0:0], \quad \rho([0:1:0])=[0:1:0],$$
$$\rho([0:0:1])=[0:0:1], \quad \rho([1:1:1])=[d:e:f].$$ 
Therefore, the  $k$-automorphism $\sigma:=\tau \circ \rho$ satisfies 
$\sigma(P_i)=Q_i$ for any $i \in \{1, 2, 3, 4\}$. 
\end{proof}

\begin{lem}\label{l-strange-unique}
Let $k$ be a field of characteristic two. 
Let $C_1$ and $C_2$ be smooth conics  in $\mathbb P^2_k$. 
Assume that there exist distinct four $k$-rational points $P_1, P_2, P_3, Q$ of $\mathbb P^2_k$ such that 
$\{P_1, P_2, P_3\} \subset C_1 \cap C_2$ and the tangent line $T_{C_i, P_j}$ of $C_i$ at $P_j$ 
passes through $Q$ for any $i \in \{1, 2\}$ and $j \in \{1, 2, 3\}$. 
Then $C_1=C_2$. 
\end{lem}

\begin{proof}
By Lemma~\ref{l-P2-auto}, we may assume that 
$$P_1=[1:0:0], \quad P_2=[0:1:0], \quad P_3=[1:1:1], \quad Q=[0:0:1].$$
It is well-known that $C_1$ and $C_2$ are strange curves (e.g.  see \cite[Theorem 1.1]{Fur14}). 
 \cite[Proposition 2.1]{Fur14} implies that for each $i\in \{1, 2\}$, 
$C_i$ is defined by a quadric homogeneous polynomial: 
$$a_ix^2+b_ixy+c_iy^2+d_iz^2 \in k[x, y, z].$$
Since $P_1, P_2, P_3 \in C_i$, we get $a_i=c_i=0$ and $b_i=d_i$. 
In particular, both of $C_1$ and $C_2$ are defined by the same polynomial $xy+z^2$. 
\end{proof}

\begin{prop}\label{p-km-surfaces}
Let $k$ be a field of characteristic two. 
Then any Keel--M\textsuperscript{c}Kernan surface $S$ of degree $3$ over $k$ 
is $k$-isomorphic to the surface $X$ constructed in Notation~\ref{n-2}. 
\end{prop}

\begin{proof}
We use the same notation as above. 
Let 
$$\pi\colon S_0 \to \mathbb P^1$$ 
be the induced $\mathbb P^1$-fibration. We divide the proof into two steps. 

\begin{step}\label{s-independent}
In this step, we show that  any two Keel--M\textsuperscript{c}Kernan surfaces $S$ and $S'$ of degree $3$ over $k$ are isomorphic over $k$. 

\medskip

There are three $k$-rational points $P_1, P_2, P_3\in C$ (resp. $P'_1, P'_2, P'_3\in C$) such that 
$S$ (resp. $S'$) is the blowup of $S_0$ twice along $P_1 \cup P_2 \cup P_3$ (resp. $P'_1 \cup P'_2 \cup P'_3$). 
Thanks to Lemma \ref{l-P2-auto}, 
there is a $k$-automorphism $\sigma\colon \mathbb P_k^2 \to \mathbb P_k^2$ such that $\sigma(Q)=Q$ and  
$\sigma(P_i)=P'_i$ for $i=1,2$ and $3$. 
Lemma \ref{l-strange-unique} implies that $\sigma(C)=C$ and, in particular, $\sigma$ induces a $k$-isomorphism $\widetilde{\sigma}:S \xrightarrow{\simeq} S'$, as desired. 
\end{step}

\begin{step}\label{s-F2}
In this step, we assume that $k=\mathbb F_2$. Note that $C$ has exactly three $\F_2$-rational points: 
$$Q_1:=[1:0:0], \quad Q_2:=[0:1:0], \quad Q_3:=[1:1:1].$$
Let
$$P_1:=\varphi^{-1}_0(Q_1), \quad P_2:=\varphi_0^{-1}(Q_2), \quad P_3:=\varphi_0^{-1}(Q_3),$$
and $S$ to be the Keel--M\textsuperscript{c}Kernan surface of degree $3$ over $\mathbb F_2$ as above. 
We now show that $S$ is $\F_2$-isomorphic to $X^{(0)}$ defined in Notation~\ref{n-2}. 

\medskip

There are pairwise disjoint $(-1)$-curves $E_1, \dots, E_7$ on $S$ over $\F_2$, 
i.e. for any $i=1, \dots, 7$, 
$E_i$ is $\F_2$-isomorphic to $\mathbb P^1_{\mathbb F_2}$ and satisfies $K_S \cdot E_i=E_i^2=-1$. 
Indeed, we can check that the following seven curves listed below satisfy these properties. 
\begin{itemize}
\item The exceptional curve over $Q$ is a $(-1)$-curve over $\F_2$,  
\item For any $i=1, 2, 3$, the exceptional curve over $Q_i$ obtained by the second blowup 
is a $(-1)$-curve over $\F_2$, and  
\item For any $1 \leq i < j \leq 3$, the proper transform of the $\F_2$-line, passing through 
$Q_i$ and $Q_j$, is a $(-1)$-curve over $\F_2$. 
\end{itemize}
Let $\psi\colon S \to T$ 
be the birational morphism with $\psi_*\MO_S=\MO_T$ that contracts $E_1, \dots, E_7$. 
Since $T$ is a projective scheme over $\F_2$ 
whose base change to $\overline{\F}_2$ is a projective plane, 
it follows that $T$ is $\F_2$-isomorphic to $\mathbb P^2_{\F_2}$ by Lemma~\ref{l-SB}. 
Thus, $S$ is obtained by the blowup along all the $\F_2$-rational points of $\mathbb P^2_{\F_2}$
which implies $S \simeq X^{(0)}$ (cf. Notation~\ref{n-2}), as desired. 
\end{step}

By Step~\ref{s-independent} and Step~\ref{s-F2}, we are done. 
\end{proof}

\appendix\def\thesection{A}
\section{Kawamata--Viehweg vanishing for smooth del Pezzo surfaces}

By Theorem~\ref{t-cex-kvv}, there exists a smooth weak del Pezzo surface of characteristic $2$
which violates Kawamata--Viehweg vanishing.  
We now show that Kawamata--Viehweg vanishing holds on smooth del Pezzo surfaces. 

\begin{prop}\label{p-A}
Let $k$ be an algebraically closed field of characteristic $p>0$. 
Let $X$ be a smooth projective surface over $k$ such that $-K_X$ is ample and let 
$(X, \Delta)$ be a klt pair for some effective $\Q$-divisor $\Delta$.
Let $D$ be a Cartier divisor such that 
$D-(K_X+\Delta)$ is nef and big. 

Then $H^i(X, \MO_X(D))=0$ for $i>0$. 
\end{prop}

\begin{proof}
After perturbing $\Delta$, we may assume that $D-(K_X+\Delta)$ is ample. 
We define $A:=D-(K_X+\Delta)$. 
We run a $(\Delta+A)$-MMP $f:X \to Y$. 
Since $-K_X$ is ample, $Y$ is also a smooth del Pezzo surface. 
Moreover, this MMP can be considered as a $(K_X+\Delta+A)$-MMP. 
By  Kawamata--Viehweg vanishing theorem for  birational morphisms 
(cf. \cite[Theorem 10.4]{Kol13}, \cite[Theorem~2.12]{Tan15}), it follows that 
$$H^i(X, \MO_X(D)) \simeq H^i(Y, f_*\MO_X(D)) \simeq H^i(Y, \MO_Y(f_*D))$$
for any $i$, where the latter isomorphism follows from the fact that $f$ is obtained by running a $D$-MMP. 

Therefore, after replacing $X$ by $Y$, 
we may assume that $\Delta+A$ is nef. 
Thus, $D-K_X$ is nef and big. 
In this case, it is well-known that $H^i(X, \MO_X(D))=0$ (e.g. see  \cite[Proposition~3.2]{Muk13} or \cite[Proposition 3.3]{CT}). 
\end{proof}



\end{document}